\documentclass[12pt]{article}
\usepackage{graphicx}
\usepackage{setspace}
\usepackage{amssymb}
\usepackage{amsmath}
\usepackage{amsthm}
\usepackage{hyperref}
\usepackage{bbm} 
\usepackage{bm} 
\usepackage{caption}
\usepackage{subcaption}
\usepackage{cite} 
\usepackage[all]{xy}
\usepackage{mathtools}

\newtheorem{theorem}{Theorem}

\newtheorem{lemma}[theorem]{Lemma}

\def\E{{\rm e}}

\begin{document}


\title{Some adjacency-invariant spaces on products of short cycles}
\author{Jeffrey A. Hogan\\
School of Mathematical and Physical Sciences\\
University of Newcastle\\
Callaghan NSW 2308 Australia\\
email: {\tt jeff.hogan@newcastle.edu.au} \and
Joseph D. Lakey\\
College of Arts and Sciences\\
New Mexico State University\\
Las Cruces, NM 88003--8001\\
email: {\tt jlakey@nmsu.edu}}

\abstract{ We study certain spaces of vertex functions on the Cayley graphs  $\mathcal{C}_m^N$ of  $\mathbb{Z}_m^N$, where $m=3,4,5$, 
that are invariant under the adjacency operator that 
maps a value at a given vertex to each of its neighbors.  An application to spatio--spectral limiting---an analogue of time and band limiting---is also discussed.
}

\maketitle

\section{Introduction}

We study  spaces of vertex functions on the Cayley graphs  $\mathcal{C}_m^N$ of  $\mathbb{Z}_m^N$, where $m=3,4,5$, 
that are invariant under the adjacency operator that 
maps a value at a given vertex to each of its neighbors.  In spectral graph theory, \emph{graph filters} are typically described
as polynomials in the adjacency operator, e.g., \cite{sandryhaila_moura_2013}. Thus, adjacency invariant spaces are ones on which filters can, in principle, admit
simple implementations. 

This work extends our work  \cite{Hogan20172,hogan2018spatiospectral} in which certain adjacency-invariant spaces on Boolean cubes $\mathcal{B}_N$---Cayley graphs on $\mathbb{Z}_2^N$---were identified.
These spaces were shown to contain eigen-spaces of \emph{spatio--spectral limiting operators}---analogues on $\mathbb{Z}_2^N$ of classical time- and band-limiting operators, e.g., \cite{hogan_lakey_tbl}
that are in turn relevant to uncertainty principles on more general graphs, e.g., \cite{benedetto_koprowski_2015,tsitsvero_etal_2015}.
Broadly, these spaces were defined in terms of level sets $\Sigma_r$ of path distance $r$ to the identity of $\mathbb{Z}_2^N$.  
We defined a \emph{subadjacency} operator
$A_+$ of the adjacency operator $A$ that maps functions supported in $\Sigma_r$ to ones supported in $\Sigma_{r+1}$, and its adjoint $A_-$. We showed that spaces of the form 
$\mathcal{V}_r=\{\sum_{k=0}^{N-r}  c_k A_+ f, f\in\mathcal{W}_r , c_0,\dots, c_{N-r}\in \mathbb{R} \,{\rm or}\, \mathbb{C}\}$, where $\mathcal{W}_r$ 
consists of those vertex functions 
$f$ supported in $\Sigma_r$ in the kernel of $A_-$, are invariant under the adjacency operator $A$.  Moreover,  $\mathcal{V}_r\simeq
\mathcal{W}_r\times \mathbb{R}^{N+1-r}$ and, as an operator on  $\mathcal{V}_r$, $A$ depends only on the factor $\mathbb{R}^{N+1-r}$
and can thus be expressed as multiplication on $\mathbb{R}^{N+1-r}$ by a \emph{ level matrix} of size $(N+1-r)$. 
The action of any polynomial graph filter on
$\mathcal{B}_N$ can then be expressed on $\mathcal{V}_r$ as a polynomial of this level matrix.  On $\mathcal{B}_N$, any vertex function can
be expressed, in any neighborhood of the identity, as a sum of components in spaces $\mathcal{V}_r$. Thus such decomposition provides 
a  tool to analyze graph filters on $\mathcal{B}_N$.

After setting notation, we establish corresponding results in $\mathcal{C}_m^N$ 
for the cases $m=3,4,5$ in Sects.~\ref{sec:c3}--\ref{sec:c5} respectively,  
showing that the corresponding adjacency matrices are invariant on suitable analogues of the spaces $\mathcal{V}_r$ outlined above.
We comment on applications to  spatio--spectral limiting in Sect.~\ref{sec:ssl}
and briefly comment on extensions to larger values of $m$ in Sect.~\ref{sec:conclusions}.

\section{Background and notation}
Let $\Gamma$ be a finite abelian group and $S$ a symmetric subset (i.e., $S=-S$) of generators of $\Gamma$, meaning that each $\gamma\in\Gamma$
can be written $\gamma=s_1+\dots+s_n$ for some $n$ with each $s_i\in S$.  
The Cayley graph $(\Gamma,S)$ is then the (unweighted, symmetric)  graph whose vertices are the elements of $\Gamma$ and whose edges correspond to ordered pairs
$(\gamma_1,\gamma_2)$ such that $\gamma_1-\gamma_2\in S$.  

In this paper we study certain properties of Cayley graphs $\mathcal{C}_m^N$ of the groups 
$\mathbb{Z}_m^N$, specifically for $m=3,4,5$.  The one-dimensional cycles  $\mathcal{C}_3, \mathcal{C}_4, \mathcal{C}_5 $ 
are depicted in Fig.~\ref{fig:short_cycles}. 
The cycle group $\mathbb{Z}_m$, the group of \emph{integers modulo $m$}, can be identified with
the set of the first $m$ integers $\{0,1,\dots, m-1\}$ with addition modulo $m$. The identity is the (equivalence class of the) integer $0$ and the inverse of $k\in \{1,\dots, m-1\}$
is  $m-k$. We will express the element $m-k$ ($k\leq \lfloor m /2\rfloor$) instead as ``$-k$'' to emphasize that it is the
$k$th-order sum of the inverse $-1$ of the generator 1.  
We write elements of $\mathbb{Z}_m^N$ as $v=(\ell_1,\dots, \ell_N)$, $\ell_i\in \{-\lfloor (m-1)/2\rfloor,\dots, \lfloor m/2\rfloor\}$. 
Addition in $\mathbb{Z}_m^N$ is coordinate-wise addition in
$\mathbb{Z}_m$.

\begin{figure}[tbhp]
\begin{centering}
\vspace{1cm}
\
$\xymatrix@C=0.1em{
-1 \ar@{-}[rr] \ar@{-}[rd] & & \ar@{-}[ld] 1 &\\
& 0 \\
}$
\qquad
$\xymatrix@C=0.1em{
&2 \ar@{-}[ld] \ar@{-}[rd] \\
-1 \ar@{-}[rd]& & 1 \ar@{-}[ld]\\
& 0 \\
}$
\qquad
$\xymatrix@C=0.1em{
-2 \ar@{-}[rr] \ar@{-}[d] & & \ar@{-}[d] 2 &\\
-1 \ar@{-}[rd]& & 1 \ar@{-}[ld]\\
& 0 \\
}$
\begin{small}
\caption{\label{fig:short_cycles} Graphic representation of cycles $\mathcal{C}_3$ (left), $\mathcal{C}_4$ (middle) and $\mathcal{C}_5$}
\end{small}
\end{centering}
\end{figure}
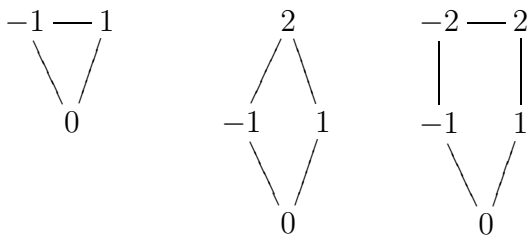

We denote by $e_k\in\mathbb{Z}_m^N$  the element equal to $1\in \mathbb{Z}_m$ in the $k$th coordinate
and equal to  $0\in \mathbb{Z}_m$ in the other coordinates. 
We denote by $\mathcal{C}_m^N$ the Cayley graph on $\mathbb{Z}_m^N$ with generators $S=\{\pm e_k\}_{k=1}^N$. Two vertices $v,w$
are adjacent, denoted $v\sim w$, if $v-w=\pm e_k$ for some $k=1,\dots, N$. 
It will be useful to define the \emph{$k$-reflection} of a vertex $v=(\ell_1,\dots, \ell_N)$ to be the vertex $\tilde{v}_k$ whose $j$th coordinate is $\ell_j$
if $j\neq k$ and whose $k$th coordinate is $-\ell_k$.   The \emph{$k$-reflection} operator $\rho_k$ is defined on vertex functions $f$
by $(\rho_k f)(v)=f(\tilde{v}_k)$.
 To an element $v=(\ell_1,\dots, \ell_N)\in  \{-\lfloor (m-1)/2\rfloor,\dots, \lfloor m/2\rfloor\}^N$ we define its
\emph{level vector} to be the vector whose $k$th entry is the $k$th \emph{level}, $d_k(v)=|\ell_k|$.  The path distance
 from $v$ to the origin is $d(v):=\sum_{k=1}^N d_k(v)$.  
  If $0<d_k(v)<\lfloor{m/2}\rfloor$ 
then $v_k^{-}$ denotes the (unique) vertex satisfying $d_k(v_k^{-})=d_k(v)-1$ and $d_j(v_k^{-})=d_j(v)$ if $j\neq k$. 
Similarly, $v_k^{+}$ denotes the (unique) vertex satisfying $d_k(v_k^{+})=d_k(v)+1$ and $d_j(v_k^{+})=d_j(v)$ if $j\neq k$.
Observe that if $v\sim w$ and $d(v)<d(w)$ then  there is a unique coordinate $k$ such that $d_k(w)=d_k(v)+1$. If $d_k(v)>0$ then
$w=v_k^+$ (and if $d_k(w)<\lfloor m/2\rfloor$ 
then $v=w_k^-$). Otherwise $d_k(v)=0$ and we regard $v_k^+=v\pm e_k$ as a pair of vertices. 
Similarly, if $d_k(w)=\lfloor m/2\rfloor$ 
and $m$ is even we also regard $w_k^-=w\mp e_k$ as a vertex pair.
As an example in $\mathcal{C}_5^2$, $(-1,1)_1^+=(-2,1)$ and $(-1,1)_1^-=(0,1)$ while $(0,1)_1^+=(\pm 1,1)$.
Two vertices $v,w$ are equi-level
if they have an equal number of coordinates at each level. For example, $v=(0,1,2,1)$ and $w=(-1,0,1,2)$ are equi-level in $\mathcal{C}_4^4$.
Set $M=\lfloor m/2\rfloor$. The level set $\Sigma_{q_1,\dots, q_{M}}$ is the set of all vertices that have $q_\ell$ coordinates at level $\ell$, $\ell=1,\dots, M$.

The graph Laplacian is the operator $(L f)(v)=\sum_{w\sim v} f(v)-f(w)$ where $f$ is a  complex-valued function defined on $\mathbb{Z}_m^N$.
When $N=1$, the eigenvectors are $E_k=\{\E^{2\pi i k\ell/m}\}_{\ell=0 }^{m-1}$, with eigenvalue $2\cos (2\pi k/N)$. 
Given the Cartesian product of two graphs with an appropriate vertex ordering on the product, the eigenvectors of the product Laplacian are 
the tensor products of eigenvectors of the Laplacians on the separate factors, whose eigenvalues are the sums of the corresponding factor eigenvalues., e.g., \cite{kurokawa2017multidimensional}.
The graph Fourier transform is the operator that maps a vertex function to its coefficients in a basis of Laplacian eigenvectors.
For any $r$-regular graph (i.e., each vertex is adjacent to $r$ other vertices), the Laplacian can also be written as $L=r I-A$ 
where the adjacency operator $A$ 
is defined by $(A f)(v)=\sum_{w\sim v} f(w)$.

Once an indexing of the $m^N$ vertices of $\mathcal{C}_m^N$ is fixed, the matrix $A$ of the adjacency operator with respect to this ordering
can be defined as the matrix whose $(i,j)$-entry is equal to one precisely when $v_i$ and $v_j$ are adjacent vertices.    In what follows we
will not distinguish notationally between the operator $A$ versus a matrix realization $A$, and simply write $A$ in either case.
The same will apply to \emph{partial} adjacencies---restrictions of adjacency to subsets of (ordered) vertex pairs---and compositions of them.
We define the \emph{outer adjacency} operator ${A}_+$
 by $({A}_+ f)(v)=\sum_{w\sim v, d(v)<d(w)} f(w)$.
If vertices
are ordered in nondecreasing order of path distance $d$ to the origin, the corresponding matrix $A_+$ of $A$ is lower triangular.
We define ${A}_-$ to be the adjoint of ${A}_+$. 
Finally, we define the \emph{neutral adjacency operator}
${A}_0$ by $({A}_0f)(v)=\sum_{w\sim v,\, d(w)=d(v)} f(w)$.

The commutator $C=  [{A}_-,\, {A}_+]={A}_- {A}_+ - {A}_+ {A}_-$
will be important. 
\begin{lemma} \label{lem:samelevel} The commutator $C=A_-A_+-A_+A_-$ has its support in pairs 
of vertices that have equal level vectors. 
\end{lemma}
Specifically, we think of ${A}_+$ as mapping values at vertices of a given level to vertices of an outer level and ${A}_-$ 
does the reverse. If $f=\bm{1}_v$ is supported at a single vertex $v\sim (\ell_1,\dots,\ell_N)$ and $w=(\ell_1',\dots,\ell_N')$ 
is an equi-level vertex such that for fixed $j\neq k$, $\ell_j'=\ell_j+1$ and $\ell_k'=\ell_k-1$, then 
$({A}_-{A}_+\bm{1}_v)(w)=({A}_+{A}_-\bm{1}_v)(w)$.

\section{Adjacency-invariant spaces on $\mathcal{C}_3^N$\label{sec:c3}}
The graph $\mathcal{C}_3^N$ is an $N$-fold product of  triangles. Let $\Sigma_r$ denote the set of vertices $v$ such that $d(v)=r$.
These are the vertices that have $\pm 1$ in $r$ coordinates and are null in the other $N-r$ coordinates. Thus $\Sigma_r$  has
$2^r\binom{N}{r}$ vertices. The identity $\sum_{r=0}^N 2^r\binom{N}{r}=3^N$ reflects that the vertices of $\mathcal{C}_3^N$ 
are the disjoint union of $\Sigma_r$, $r=0,\dots, N$. The outer adjacency operator maps vectors supported on $\Sigma_r$ to ones
supported on $\Sigma_{r+1}$. In particular, $A_+^{k}f=0$ if $f$ is supported in $\Sigma_r$ and $k>N-r$. 
Two vertices  $v,w\in\Sigma_r$ that are equal in all but one coordinate are adjacent to one another. 
Specifically, if $v,w\in\Sigma_r$ and $v-w=\pm e_k$ 
then  $v\sim w$ and $w=\tilde{v}_k$. We call $w$ the $k$th neutral neighbor of $v$. The neutral adjacency operator on $\mathcal{C}_3^N$
is $(A_0 f)(v)=\sum_{k: d_k(v)=1} f(\tilde{v}_k)$.

\begin{theorem} \label{thm:commutator_3} ${C}+{A}_0$ is diagonal on the space of vectors supported on $\Sigma_r$, the vertices of distance $r$ from the origin in $\mathcal{C}_3^N$.
Specifically, for such $f$, 
${C}f=(2N-3r) f-{A}_0 f$.
\end{theorem}

\begin{proof} Fix $v\in \Sigma_r$. In view of Lem.~\ref{lem:samelevel}, ${C}f(v)$ takes its input from vertices that differ from $v$ in at most a single coordinate, 
meaning ${C}$ will have a diagonal component at $v$
and, possibly, a neutral neighbor component. Given a vertex $v$, for each null coordinate of $v$, there are two edges to vertices in $\Sigma_{r+1}$ that are level one in that coordinate and otherwise equal to $v$, so there are $2(N-r)$
paths from $v$ to a vertex in $\Sigma_{r+1}$ and back to $v$ under ${A}_-{A}_+$. On the other hand, for each level-one coordinate 
$k$ of $v$, there is a path from $v$ to the vertex $v_k^-$ that is null in that coordinate and equal to $v$ in all other coordinates,
then back to $v$, and another path through $v_k^-$ 
to $\tilde{v}_k$, under ${A}_+{A}_-$.
Altogether then, ${C}(f)(v)=({A}_-{A}_+-{A}_+{A}_-)(f)(v)= 2(N-r) f(v)
-[r f(v)+\sum_{k: d_k(v)=1} f(\tilde{v}_k)]=(2N-3r)f(v)-({A}_0 f)(v)$. This proves the theorem.
\end{proof}

\begin{lemma}  Let $f$ be a vertex function supported in $\Sigma_r$ such that for each vertex $v\in\Sigma_r$, 
$f$ is $k$-symmetric in $s$ of the level-one coordinates of $v$
and $k$-antisymmetric in the other $r-s$ level-one coordinates. Then $f$ is an eigenvector of ${A}_0$ with eigenvalue
$2s-r$.
\end{lemma}

\begin{proof} Each vertex  $v\in \Sigma_r$ has $r$ coordinate neutral neighbors, namely $\{\tilde{v}_k: d_k(v)=1\}$ and 
\begin{equation*}({A}_0f)(v)=\sum_{k:d_k(v)=1} f(\tilde{v}_k)=s f(v)-(r-s) f(v)=(2s-r) f(v)\, 
\end{equation*}
by hypothesis.
\end{proof}

The possible $A_0$-eigenvalues $\lambda$ of $f$ on $\Sigma_r$ are $(2s-r)$, $s=0,1\dots, r$. 

\begin{lemma}   \label{lem:eigen_outer_3} Let $f$ be a $\lambda$-eigenvector of ${A}_0 $ supported in $\Sigma_r$. 
Then ${A}_+ f$ is a $(\lambda+1)$-eigenvector of ${A}_0 $ supported in  $\Sigma_{r+1}$. 
\end{lemma}

\begin{proof} One has $(A_0f)(v)=\sum_{k:d_k(v)=1} f(\tilde{v}_k)$.  Since $A_+f(w)=\sum_{\nu:d_\nu(w)=1} f(w_\nu^-)$,
altogether,
\begin{multline*}(A_0A_+f)(w)=\sum_{\mu:d_\mu(w)=1}(A_+f)(\tilde{w}_\mu)
=\sum_{\mu:d_\mu(w)=1}\sum_{\nu:d_\nu(\tilde{w}_\mu)=1} f((\tilde{w}_\mu)_\nu^-)\\
=\sum_{\mu:d_\mu(w)=1}\sum_{\nu:d_\nu(w)=1}f((\tilde{w}_\mu)_\nu^-)
=\sum_{\nu:d_\nu(w)=1} \sum_{\mu:d_\mu(w)=1}f((\tilde{w}_\mu)_\nu^-)\\
=\sum_{\nu:d_\nu(w)=1}[ f((\tilde{w}_\nu)_\nu^-)+\sum_{\mu:d_\mu(w)=1,\,\mu\neq \nu }f((\tilde{w}_\mu)_\nu^-)]\\
=\sum_{\nu:d_\nu(w)=1}[ f(w_\nu^-)+(A_0f)(w_\nu^-)]=(\lambda+1)\sum_{\nu:d_\nu(w)=1} f(w_\nu^-)=(\lambda+1) (A_+f)(w)
\end{multline*}
 since $d_\nu(w_\nu^-)=0$ implies that $\{(\tilde{w}_\mu)_\nu^-: d_\mu(w)=1,\, \mu\neq \nu\}$ is the set of all level-one reflections
of $w_\nu^-$ and that $f((\tilde{w}_\nu)_\nu^-)=f(w_\nu^-)$. We also used that $A_0f=\lambda f$.
\end{proof}

In $\mathcal{C}_3^N$, $\Sigma_r$ has $2^r\binom{N}{r}$ vertices. For each choice $S_r$ of $r$ nonzero coordinates, 
a basis for the vertex functions whose nonzero coordinates are $S_r$ can be obtained by ordering the elements of $S_r$
and forming the Hadamard matrix $H_r$ of size $2^r$ (i.e., the $r$-th Kronecker power of the Haar matrix with 
rows $[1,1]/\sqrt{2}$ and $[1,-1]/\sqrt{2}$) on these indices.  For each such $S_r$, there is an $\binom{r}{s}$-dimensional
eigenspace of $A_0$ with eigenvalue $2s-r$ spanned by the columns of $H_r$ that are symmetric in $s$ of the $S_r$-coordinates
and antisymmetric in the others.
A basis for the $(2s-r)$-eigenspace
of $A_0$ on $\Sigma_r$ results from ranging over choices of $S_r$.

\begin{lemma} \label{lem:eigen_outer_power_3} Let $f$ be a $\lambda$-eigenvector of $A_0$ supported in $\Sigma_r$ and in the kernel of $A_-$. Then
$A_- A_+^{k+1} f =m(r,k,\lambda) A_+^k f$ where $m(r,k,\lambda)$ is defined inductively by $m(r,0,\lambda)=2N-3r-\lambda$
and $m(r,k,\lambda)=m(r,k-1,\lambda)+(2N-3r-4k)-\lambda$.
\end{lemma}

If $A_+^kf=0$ for some $k$ then $A_+^{k+1}f=0$ also. The identity reflects this property.

\begin{proof} First, since $A_-f=0$ we have $Cf=A_-A_+f-A_+A_-f=A_-A_+ f=((2N-3r)-\lambda)f$ by Thm.~\ref{thm:commutator_3}.
Thus the result holds with $k=0$ and $m(r,0,\lambda)=(2N-3r-\lambda)$.  Suppose that the result is true for $k-1$.
Then, again by  Thm.~\ref{thm:commutator_3}, 
\begin{multline*} A_-A_+^{k+1} f=[A_-A_+-A_+A_-+A_+A_-]A_+^k f= CA_+^k f+A_+(A_-A_+^k f)\\
=CA_+^k f+m(r,k-1,\lambda) A_+^kf=(2N-3(r+k)) A_+^k f -A_0A_+^k f +m(r,k-1,\lambda) A_+^k f \\
=[(2N-3(r+k))-(\lambda+k)+m(r,k-1,\lambda)]A_+^k f  = m(r,k,\lambda)A_+^k f \qedhere
\end{multline*}
\end{proof}

\begin{theorem} \label{thm:adjacency_invariant_C3} Let $\mathcal{W}_{r,\lambda}$ denote the space of vertex functions supported in $\Sigma_r$ that are in the kernel of ${A}_-$
and are $\lambda$-eigenvectors of ${A}_0$. 
Let  $\mathcal{V}_{r,\lambda}=\{\sum_{k=0}^{N-r} c_k {A}_+^k f: f\in \mathcal{W}_{r,\lambda},\, c_k\in\mathbb{C}\}$.
Then for each eigenvalue $\lambda$ of ${A}_0$, the space  $\mathcal{V}_{r,\lambda}$ is invariant under the adjacency operator $A$
on $\mathcal{C}_3^N$.
\end{theorem} 

\begin{proof} By its definition, $\mathcal{V}_{r,\lambda}$ is invariant under $A_+$.  By Lem.~\ref{lem:eigen_outer_3}, 
$A_+^k f$ is an eigenvector of $A_0$ for each $k$ if $f\in \mathcal{W}_{r,\lambda}$\ and therefore $A_0$ maps any sum of the form 
$\sum_{k=0}^{N-r} c_k {A}_+^k f: f\in \mathcal{W}_{r,\lambda}$ to another sum of the same form. Finally, by Lem.~\ref{lem:eigen_outer_power_3}, 
 $A_-$ maps any sum of this form to another such sum. Since $A=A_-+A_0+A_+$, the theorem follows.
\end{proof}

Evidently,  $\mathcal{V}_{r,\lambda}$ is isomorphic to $\mathcal{W}_{r,\lambda}\times \mathbb{C}^{N+1-r}$.
The action of $A$ on $\mathcal{V}_{r,\lambda}$ only depends on the $\mathbb{C}^{N+1-r}$ factor and can be represented by 
a size $N+1-r$ tri-diagonal \emph{level matrix} that has ones on the diagonal below the main diagonal, entries $\lambda+k$ on the main diagonal,
and $m(r,k,\lambda)$ on the diagonal above the main diagonal. 
This representation can make application of polynomials of $A$ on  $\mathcal{V}_{r,\lambda}$ feasible when $N$ is large, cf.,
\cite{hogan2018spatiospectral}. 

\section{Adjacency-invariant spaces on $\mathcal{C}_4^N$\label{sec:c4}}

In this case, $\Sigma_r=\{v:\, d(v)=r\}=\cup_{2q+p=r} \Sigma_{p,q}$ where $\Sigma_{p,q}$ consists of those vertices
that have $q$ level-two coordinates and $p$ level-one coordinates. For example $\Sigma_5=\Sigma_{5,0}\cup\Sigma_{3,1}\cup\Sigma_{1,2}$.

\begin{theorem} \label{thm:commutator_4} 
Let $f$ be supported on  $\Sigma_r$, the vertices of distance $r$ from the origin in $\mathcal{C}_4^N$. Then ${C}f=2(N-r) f$.
\end{theorem}

\begin{proof} It suffices to show that the same multiplier applies to $\Sigma_{p,q}$ for each case of $p+2q=r$.  
We write $A_+$ and $A_-$ explicitly  in this case as
$(A_+f)(v)=\sum_{d_k(v)=2} f(v_k^-) +\sum_{d_k(v)=1} f(v_k^-)$ and $(A_- g)(w)=\sum_{d_k(w)=0} f(w_k^+)+\sum_{d_k(w)=1} f(w_k^+)$.
Thus 
\begin{multline*}(A_-A_+ f)(v) =\Bigl(\sum_{d_k(v)=0}+\sum_{d_k(v)=1}\Bigr) (A_+f)(v_k^+) \\
=\Bigl(\sum_{d_k(v)=0}+\sum_{d_k(v)=1}\Bigr)\Bigr( \sum_{d_\nu(v_k^+)=2} +\sum_{d_\nu(v_k^+)=1}\Bigr) f((v_k^+)_\nu^-)\\
=\Bigl(\sum_{d_k(v)=0}+\sum_{d_k(v)=1}\Bigr) f((v_k^+)_k^-)+{\rm other \,\, terms}
\end{multline*}
Similarly,
\begin{multline*}(A_+A_- f)(v) =\Bigl(\sum_{d_\nu(v)=1}+\sum_{d_\nu(v)=2}\Bigr) (A_-f)(v_\nu^-) \\
=\Bigl( \sum_{d_\nu(v)=2} +\sum_{d_\nu(v)=1}\Bigr)\Bigl(\sum_{d_k(v_\nu^-)=0}+\sum_{d_k(v_\nu^-)=1})  f((v_\nu^-)_k^+)\Bigr) \\
=\Bigl( \sum_{d_\nu(v)=2} +\sum_{d_\nu(v)=1}\Bigr)  f((v_\nu^-)_\nu^+)+{\rm other \,\, terms}
\end{multline*} 
Contributions to $(A_\pm A_\mp f)(v)$ come from two-edge paths of two types:  paths in which the level changes are in the same coordinate
and paths that involve a level increase in one coordinate and a level
decrease in a different coordinate. The latter account for the \emph{other terms} above.
As observed in Lem.~\ref{lem:samelevel}, these \emph{other terms}
cancel between $A_-A_+$ and $A_+A_-$.  Also, if $d_k(v)=1$ then the $k$th coordinate
contribution to 
$A_-A_+f(v)$ is $f(v)+f(\tilde{v}_k)$, since $(v_k^+)_k^-$ is the pair $(v, \tilde{v}_k)$.  
The $k$th coordinate contribution to  $A_+A_-f(v)$
is also $f(v)+f(\tilde{v}_k)$,
since $(v_k^-)_k^+$ is also the pair $(v,\tilde{v}_k)$, so the \emph{level-one} terms  cancel in $Cf$.
If $d_k(v)=0$ then $v_k^+$ and $\widetilde{(v_k^+)}_k$ are both assigned the value $f(v)$ under $A_+$ so 
the $k$th coordinate contribution to $(A_-A_+f)(v)$ is $2f(v)$ whereas the $k$th contribution to $(A_+A_-f)(v)$ is zero,
so the net contribution to $Cf$ from level-zero terms is $2(N-(p+q)) f$. 
On the other hand, if $d_\nu(v)=2$  then the  $\nu$th coordinate contribution to $(A_-A_+f)(v)$ is zero, whereas  
$v_\nu^-$ and $(\widetilde{v_\nu^{-}})_\nu$ are both assigned the value $f(v)$ under $A_-$, so the $v_\nu$-th contribution to 
$(A_+A_-f)(v)$ is $2f(v)$, and the net level-two contribution to $Cf $ is $-2q f$.
Altogether, the different coordinate contributions to $(Cf)(v)$ add up to $2(N-(p+q)-q) f(v)=2(N-r) f(v)$, independent of the level component, 
$\Sigma_{p,q}\subset \Sigma_r$, of $v$. This proves the theorem.
\end{proof}

\begin{lemma} Suppose that $f$ is supported in $\Sigma_{p,q}$ and $A_-f=0$. Then $(A_-A_+^{k+1} f)(v)=m(r,k) f(v)$ where
$m(r,k)$ is defined iteratively by $m(r,0)=2(N-r)$ and $m(r,k)=[m(r,k-1)+2(N-(r+k))]$, $r=p+2q$.
\end{lemma}

\begin{proof} The case $k=0$ follows from Thm.~\ref{thm:commutator_4} and the hypothesis that $f$ is in the kernel of $A_-$.
Suppose that the conclusion holds for $k$ replaced by $k-1$. Then
\begin{multline*} A_-A_+^{k+1} f= (A_-A_+-A_+A_-+A_+A_-)A^kf =C A_+^k f+A_+A_- A_+^{k} f\\
=2(N-(r+k)) A_+^k f +A_+ (m(r,k-1)) A_+^{k-1} f = m(r,k) A_+^k f\, .
\end{multline*}
Here we used again Thm.~\ref{thm:commutator_4}. It applies because $A_+^k f$ is supported in $\Sigma_r=\cup_{p+2q=r}\Sigma_{p,q}$,
and then applied the induction hypothesis. This proves the lemma.
\end{proof}

\begin{theorem} \label{thm:adjacency_invariant_C4} Let $\mathcal{W}_{p,q,0}$ be the space of vertex functions supported in $\Sigma_{p,q}$ and in the kernel of $A_-$.
Set $\mathcal{V}_{p,q}=\{\sum_{k=0}^{N-r} c_k A_+^k f,\, f\in \mathcal{W}_{p,q,0}\}$. Then $\mathcal{V}_{p,q}$ is invariant under $A$.
\end{theorem}

\begin{proof} Observe that there are no neutral adjacencies, that is, vertex pairs $v,w\in \Sigma_r$ such that $v\sim w$. Thus $A=A_++A_-$.
Evidently $\mathcal{V}_{p,q}$ is invariant under $A_+$.  If $g=\sum_{k=0}^{N-r} c_k A_+^k f$, $f\in \mathcal{W}_{p,q,0}$ then 
\[A_- g=\sum_{k=1}^{N-r}
c_k m(r,k-1) A_+^{k-1} f =\sum_{k=0}^{N-r-1} c_{k+1} m(r,k) A_+^k f\] 
which is also in $\mathcal{V}_{p,q}$ (with no term $A_+^{N-r} f$).
\end{proof}

Similar to the case of $\mathcal{C}_3^N$,  $\mathcal{V}_{p,q}\simeq \mathcal{W}_{p,q,0}\times \mathbb{C}^{N+1-r}$ ($r=p+2q$),
and the action of $A$ on $\mathcal{V}_{p,q}$ can be realized as a tridiagonal level matrix on $ \mathbb{C}^{N+1-r}$.
In this case the matrix has ones on the diagonal below the main diagonal, $m(r,k)$ on the diagonal above the main diagonal, and zeros on the 
main diagonal.

\section{Adjacency-invariant spaces on $\mathcal{C}_5^N$\label{sec:c5}}

As in the $\mathcal{C}_3^N$ case, $\mathcal{C}_5^N$
has \emph{neutral adjacencies}, this time because ``$-2$'' and ``$2$'' are adjacent in $\mathcal{C}_5$, see Fig.~\ref{fig:short_cycles}.
Additionally, a path from ``$1$'' to ``$-1$'' through the origin in $\mathcal{C}_5$ can be viewed as a reflection, which arises under composite
partial adjacencies $A_+A_-$. 
These properties translate into more technical description of adjacency invariant  spaces on $\mathcal{C}_5^N$, to which we will refer
as $\mathcal{V}_{p,q}$-type spaces, 
compared to the invariant spaces $\mathcal{V}_{r,\lambda}$ on $\mathcal{C}_3^N$ and $\mathcal{V}_{p,q}$ on $\mathcal{C}_4^N$. To facilitate such description,
we in turn describe \emph{subadjacencies} $A_{(p,q)\to (p-1,q+1)}$  and $A_{(p,q)\to (p+1,q)}$ of $A_+$ 
that map values on vertices  $v\in \Sigma_{p,q}$ such that $d(v)=p+2q$ to corresponding values 
on adjacent vertices $w$ in $\Sigma_{p-1,q+1}$ or $\Sigma_{p+1,q}$, respectively, such that $d(w)=p+2q+1$. Similarly, we denote by 
  ${A}_{(p,q)\to (p-1,q)}$ and  ${A}_{(p,q)\to (p+1,q-1)}$  subadjacencies of ${A}_-$ that map 
 $\Sigma_{p,q}$ to $\Sigma_{p-1,q}$ and $\Sigma_{p+1,q-1}$, respectively.
Finally, we denote by  ${A}_{(p,q)\to (p,q)}$  the \emph{neutral subadjacency} operator on $\Sigma_{p,q}$ that replaces
a vertex function value $f(v)$ at $v\in \Sigma_{p,q}$ by the sum $\sum_{k:d_k(v)=2} f(\tilde{v}_k) $ of the values at its neutral neighbors in $\Sigma_{p,q}$. 
For $q=1$, $A_{(p,1)\to (p,1)}^2=I$. More generally, ${A}_{(p,q)\to (p,q)}$ is in the kernel of a nontrivial polynomial of degree $q+1$.
We refer to the subadjacencies from  $\Sigma_{p,q}$ to $\Sigma_{p\pm 1,q\mp1}$ 
as \emph{internal adjacencies},
referring to raising and lowering the level in a non-null coordinate, and to ${A}_{(p,q)\to (p\pm 1,q)}$
as \emph{external adjacencies} that involve adjacencies between vertices with different numbers of null coordinates.
We will use the symbols $A_-,A_+,A_0$ as abbreviations for particular inner, outer, and neutral subadjacencies when appropriate.

In defining suitable analogues on $\mathcal{C}_5^N$ of the adjacency-invariant 
subspaces of type $\mathcal{V}_{p,q}$ defined in Thms.~\ref{thm:adjacency_invariant_C3} and  \ref{thm:adjacency_invariant_C4},
between successive powers of $A_+$ that occur in their description, we need to allow application of a neutral adjacency after every subadjacency of the form 
$A_{(p,q)\to (p-1,q+1)}$ that increases
the number of level-two coordinates. So elements of $\mathcal{V}_{p,q}$-type spaces will be sums of
compositions of outer and neutral subadjacencies  starting from a fixed level.
The different subadjacencies in turn give rise to  different multipliers for each different level. 
As before, establishing adjacency invariance of $\mathcal{V}_{p,q}$-type spaces boils down to understanding compositions of the different
subadjacencies.

The commutator $C=[{A}_-,{A}_+]$ acting on vectors supported in $\Sigma_{p,q}$ breaks into components as follows.
\begin{multline*}\left. C\right|_{\Sigma_{p,q}}={A}_{(p+1,q)\to (p,q)} {A}_{(p,q)\to (p+1,q)} +{A}_{(p-1,q+1)\to (p,q)} {A}_{(p,q)\to (p-1,q+1)} \\
+{A}_{(p+1,q)\to (p+2,q-1)} {A}_{(p,q)\to (p+1,q)}  
+{A}_{(p-1,q+1)\to (p-2,q+1)} {A}_{(p,q)\to (p-1,q+1)}\\
-{A}_{(p-1,q)\to (p,q)} {A}_{(p,q)\to (p-1,q)} -{A}_{(p+1,q-1)\to (p,q)} {A}_{(p,q)\to (p+1,q-1)} \\
-{A}_{(p-1,q)\to (p-2,q+1)} {A}_{(p,q)\to (p-1,q)} - 
{A}_{(p+1,q-1)\to (p+2,q-1)} {A}_{(p,q)\to (p+1,q-1)} \\
=[  {A}_{(p+1,q)\to (p,q)} {A}_{(p,q)\to (p+1,q)} -{A}_{(p-1,q)\to (p,q)} {A}_{(p,q)\to (p-1,q)} ]\\
+[{A}_{(p-1,q+1)\to (p,q)} {A}_{(p,q)\to (p-1,q+1)} -{A}_{(p+1,q-1)\to (p,q)} {A}_{(p,q)\to (p+1,q-1)}  ]\\
+[  {A}_{(p-1,q+1)\to (p-2,q+1)} {A}_{(p,q)\to (p-1,q+1)} -{A}_{(p-1,q)\to (p-2,q+1)} {A}_{(p,q)\to (p-1,q)}]\\
+[{A}_{(p+1,q)\to (p+2,q-1)} {A}_{(p,q)\to (p+1,q)}  -{A}_{(p+1,q-1)\to (p+2,q-1)} {A}_{(p,q)\to (p+1,q-1)}],
\end{multline*}
expressing the commutator on $\Sigma_{p,q}$ as a sum of external and internal commutators with targets in $\Sigma_{p,q}$, $\Sigma_{p-2,q+1}$ and 
$\Sigma_{p+2,q-1}$ respectively.  The terms that map data on $\Sigma_{p,q}$ to $\Sigma_{p-2,q+1}$ and to $\Sigma_{p+2,q-1}$
vanish because of Lem.~\ref{lem:samelevel}.   That lemma also implies that all vertex-pair terms affected in 
\[ {A}_{(p+1,q)\to (p,q)} {A}_{(p,q)\to (p+1,q)}-{A}_{(p-1,q)\to (p,q)} {A}_{(p,q)\to (p-1,q)}  
\]
vanish except those in which the extension and reduction occurs in the same coordinate. There are $N-(p+q)$ null coordinates in 
any $\Sigma_{p,q}$-vertex, hence $2[N-(p+q)]$ \emph{out-and-back paths} to any vertex in $\Sigma_{p,q}$ via this term, while there are $p$ \emph{in-and-back}
paths plus an \emph{in-and-out} path to each level-one vertex reflection. Consequently the action of this term on a vertex function $f$ supported
in $\Sigma_{p,q}$ is multiplication by $2(N-q)-3p$, plus ${R}_1$, the level-one reflection of $f$, defined by $(R_1f)(v)=\sum_{d_k(v)=1} (\rho_k f)(v)$. 
Finally, Lem.~\ref{lem:samelevel} also implies that
any vertex-pair term in
\[{A}_{(p-1,q+1)\to (p,q)} {A}_{(p,q)\to (p-1,q+1)}- {A}_{(p+1,q-1)\to (p,q)} {A}_{(p,q)\to (p+1,q-1)} 
\]
also vanishes except for those in which the extension and reduction ends in the same coordinate. However, in this case all extensions are level one to level two and all reductions
are level two to level one. Thus the in-coordinate elements of both terms ${A}_{(p+1,q-1)\to (p,q)} {A}_{(p,q)\to (p+1,q-1)}$ and ${A}_{(p-1,q+1)\to (p,q)} {A}_{(p,q)\to (p-1,q+1)}$
map a vertex value to itself, so the two terms cancel and this term of the commutator vanishes. Altogether we have proved the following.

\begin{theorem}  \label{thm:commutator_C5} As an operator on the space of vectors supported on $\Sigma_{p,q}$, 
the vertices that have $p$-coordinates at level one and 
$q$ coordinates at level two in $\mathcal{C}_5^N$, $ C=[A_-,\, A_+] =(2(N-q)-3p) I+R_1$.
\end{theorem}

Adjacency invariant space analogues of those presented for  $\mathcal{C}_3^N$ and $\mathcal{C}_4^N$ in Thms.~\ref{thm:adjacency_invariant_C3} and  \ref{thm:adjacency_invariant_C4}
are necessarily more complicated in the case of $\mathcal{C}_5^N$. As in $\mathcal{C}_3^N$, there is a role played by 
eigenvectors of neutral/reflection maps.

\begin{lemma} \label{lem:extension-reflection} $({\rho}_k {A}_+ f)(w)= ( {A}_+ {\rho}_k f)(w)$ and $({\rho}_k {A}_- f)(w)= ( {A}_- {\rho}_k f)(w)$
\end{lemma}

\begin{proof} We prove that $({\rho}_k {A}_+ f)(w)= ( {A}_+ {\rho}_k f)(w)$  and leave $({\rho}_k {A}_- f)(w)= ( {A}_- {\rho}_k f)(w)$
as an exercise.
Observe that on the right, the reflection is in vertices in $\Sigma_{p,q}$ and on the left the reflection is in vertices
in $\Sigma_{p+1,q}$ or $\Sigma_{p-1,q+1}$. Since ${A}_+$ can be expressed as ${A}_{(p,q)\to (p+1,q)}+{A}_{(p,q)\to (p-1,q+1)}$,
it suffices to prove the commutation separately for these.

First, let $w\in \Sigma_{p-1,q+1}$ with $d_k(w)=2$. Observe that 
\begin{multline*}({A}_{(p,q)\to (p-1,q+1)} {\rho}_k f)(w)
=\sum_{\nu: d_\nu(w)=2} ({\rho}_k f)(w_\nu^{-})=\sum_{\nu: d_\nu(w)=2} f((w_\nu^{-})_k^{\sim}) \\
=\sum_{\nu: d_\nu(\tilde{w}_k)=2} f((\tilde{w}_k)_\nu^{-}) = ({A}_{(p,q)\to (p-1,q+1)}  f)(\tilde{w}_k)
=({\rho}_k {A}_{(p,q)\to (p-1,q+1)}  f)(w)\, .
\end{multline*}
Here we used the fact that $(\tilde{w}_k)_\nu^{-}=(w_\nu^{-})_k^{\sim}$. When $k\neq \nu$ the $k$th coordinate of both sides
is level two. When $k=\nu$ the $k$th coordinate of both sides is level one and the identity reduces to the fact that 
reduction commutes with reflection in each coordinate.
The identity in Lem.~\ref{lem:extension-reflection} clearly holds when $d_k(w)\in\{0,1\}$ since in these cases the reduction coordinate 
$\nu$ ($d_\nu(w)=2$) in the sum defining outer adjacency is never equal to the
reflection coordinate $k$.

Next let $w\in \Sigma_{p+1,q}$ such that $d_k(w)=1$. Observe now that 
\begin{multline*}({A}_{(p,q)\to (p+1,q)} {\rho}_k f)(w)
=\sum_{\nu: d_\nu(w)=1} ({\rho}_k f)(w_\nu^{-})=\sum_{\nu: d_\nu(w)=1} f((w_\nu^{-})_k^{\sim}) \\
=\sum_{\nu: d_\nu(\tilde{w}_k)=1} f((\tilde{w}_k)_\nu^{-})=
({A}_{(p,q)\to (p+1,q)}  f)(\tilde{w}_k)
=({\rho}_k {A}_{(p,q)\to (p+1,q)}  f)(w)\, .
\end{multline*}
Here, again, $(\tilde{w}_k)_\nu^{-}=(w_\nu^{-})_k^{\sim}$: when $k\neq \nu$ the $k$th coordinate of both sides
is level one. When $k=\nu$ the $k$th coordinate of both sides is level zero so the reflection is trivial.
The identity in Lem.~\ref{lem:extension-reflection} clearly holds when $d_k(w)\in\{0,2\}$
 since in these cases the reduction coordinate $\nu$ is never equal to the
reflection coordinate $k$.
\end{proof}

\begin{lemma}\label{lem:eigen_reflection_5} If $f$  supported in $\Sigma_{p,q}$ is a $\lambda$-eigenvector
of $R_1$ then (i) $A_{(p,q)\to (p+1,q)}f$ is a $(\lambda+1)$-eigenvector of $R_1$ on $\Sigma_{p+1,q}$ and  
(ii) $A_{(p,q)\to (p-1,q+1)}f$ is a $\lambda$-eigenvector of $R_1$ on $\Sigma_{p-1,q+1}$.
\end{lemma}

\begin{proof} First we consider the case of $A_{(p,q)\to (p+1,q)}f$.
Suppose that $w\in\Sigma_{p+1,q}$, and $d_k(w)=1$ so $w_k^-=\tilde{w}_k^-$. 
Then 
\[{\rho}_k (A_+ f)(w)
= {\rho}_k f(w_k^-)+\sum_{\nu\neq k, d_\nu(w)=1} {\rho}_k f(w_\nu^-)  
=  f(w_k^-)+\sum_{\nu\neq k, d_\nu(\tilde{w}_k)=1}  f(\widetilde{(w_\nu^-)}_k) \, .
\]
 Summing over all $k$ such that $d_k(w)=1$ then gives 
\begin{multline*}\sum_{k:d_k(w)=1} {\rho}_k (A_+ f)(w) 
= \sum_{k:d_k(w)=1} f(w_k^-)+\sum_{k:d_k(w)=1} \sum_{\nu\neq k, d_\nu(\tilde{w}_k)=1}  f(\widetilde{(w_\nu^-)}_k)  
 \\
= (A_+ f)(w)\, +\sum_{\nu:\, d_\nu(w)=1} \sum_{k\neq \nu :d_k(w)=1}   f(\widetilde{(w_\nu^-)}_k)
 \\
= (A_+ f)(w)\, +\sum_{\nu:\, d_\nu(w)=1} R_1 f(w_\nu^-)\, .
\end{multline*}
If $f$ is a  $\lambda$-eigenvector of  $R_1$ on $\Sigma_{p,q}$ then the last term becomes
\[ (A_+ f)(w) + \sum_{\nu:\, d_\nu(w)=1} \lambda f(w_\nu^-)
= \lambda (A_+ f)(w) +(A_+f)(w)=(\lambda+1)A_+f(w)\, .
\]
This gives the desired result  for $A_{(p,q)\to (p+1,q)}$.

For the case $A_{(p,q)\to (p-1,q+1)}$, let $w\in \Sigma_{p-1,q+1}$ such that $d_k(w)=1$. 
 Then $k$ cannot be one of the reduction coordinates defining $(A_+f)(w)$, so
\begin{multline*}{\rho}_k (A_+ f)(w)={\rho}_k\Bigl[\sum_{d_\nu(w)=2} f(w_\nu^-)\Bigr] = \sum_{d_\nu(w)=2}{\rho}_k  f(w_\nu^-)  
=\sum_{ d_\nu(\tilde{w}_k)=2}  f(\widetilde{(w_\nu^-)}_k) \, . \\
\end{multline*}
 Summing over all $k$ such that $d_k(w)=1$, 
\begin{multline*}\sum_{k:d_k(w)=1} {\rho}_k (A_+ f)(w) = \sum_{k:d_k(w)=1} \sum_{ d_\nu(\tilde{w}_k)=2}  f(\widetilde{(w_\nu^-)}_k)   
= \sum_{\nu:\, d_\nu(w)=2} \sum_{k :d_k(w)=1}   f(\widetilde{(w_\nu^-)}_k)  
\, \\
= \sum_{\nu:\, d_\nu(w)=2} R_1 f(w_\nu^-) = {A}_+ (R_1f)(w) =\lambda  ({A}_+ f)(w)
\, 
\end{multline*}
if $f$ is a $\lambda$-eigenvector of $R_1$. Thus $A_{(p,q)\to (p-1,q+1)}$ preserves eigenvalues of $R_1$.
\end{proof}
As a further observation, each coordinate reflection $\rho_k$ commutes with $R_1$, the sum of the level-one coordinate reflections.
Consequently, coordinate reflections preserve eigenvalues of $R_1$.
Thms.~\ref{thm:commutator_3} and  \ref{thm:commutator_4} played crucial roles in the inductive steps of the proofs of the 
Thms.~\ref{thm:adjacency_invariant_C3} and  \ref{thm:adjacency_invariant_C4} describing adjacency invariant spaces on
$\mathcal{C}_3^N$ and $\mathcal{C}_4^N$, as they allowed to interchange inner and outer adjacencies. In $\mathcal{C}_5^N$,
the corresponding spaces of type $\mathcal{V}_{p,q}$ also involve neutral adjacency maps at different levels.
To establish invariance of the corresponding spaces under inner adjacencies, we need a similar method to interchange inner
and neutral adjacencies. 

On $\mathcal{C}_5^N$, the neutral adjacency operator is $(A_0 f)(w)=\sum_{d_k(w)=2} f(\tilde{w}_k)$. 
When $A_+=A_{(p,q)\to (p-1,q+1)}$,
 $A_+f(w)=\sum_{d_\nu(w)=2} f(w_\nu^{-})$, so 
\begin{multline*}A_0A_+f(w)=\sum_{d_k(w)=2} A_+f(\tilde{w}_k)=\sum_{d_k(w)=2}  \sum_{d_\nu(w)=2} f((\tilde{w}_k)_\nu^{-})\\
= \sum_{d_\nu(w)=2}  \sum_{d_k(w)=2} f((\tilde{w}_k)_\nu^{-}) 
= \sum_{d_\nu(w)=2} [ f((\tilde{w}_\nu)_\nu^{-}) +\sum_{k\neq \nu, d_k(w)=2} f((\tilde{w}_k)_\nu^{-})] \\
= \sum_{d_\nu(w)=2} [ f((\tilde{w}_\nu)_\nu^{-}) +(A_0 f)(w_\nu^{-})]
=(A_+ A_0 f)(w)+\sum_{d_\nu(w)=2}  f((\tilde{w}_\nu)_\nu^{-}) \, .
\end{multline*}
We refer to the commutator 
$A_0A_+-A_+A_0$ in the form $A_+=A_{(p,q)\to (p-1,q+1)}$ that maps
$f$ to $(\widetilde{A_+}f)(v):=
\sum_{\nu:d_\nu(v)=2} (\rho_\nu f)(v_\nu^-)$ as \emph{twisted outer adjacency} since, without the reflections,
$(A_+f)(v)=\sum_{\nu:d_\nu(v)=2} f(v_\nu^-)$.

If 
$ g(w)= \widetilde{A_+} f(w)=\sum_{d_\nu(w)=2}  f((\widetilde{w_\nu^{-}})_\nu) $, $w\in \Sigma_{p-1,q+1}$, then
\begin{equation}\label{eq:aminusg1} (A_-g)(v)=\sum_{\sum_{\mu: d_\mu(v)=1}}\sum_{\nu: d_\nu(v_\mu^+)=2}  f((\widetilde{(v_\mu^+)_\nu^-})_\nu),\quad v\in\Sigma_{p,q}\, .
\end{equation}
But
\[ \sum_{\mu: d_\mu(v)=1} \sum_{d_\nu(v_\mu^+)=2}  f((\widetilde{(v_\mu^+)_\nu^-})_\nu) =
\sum_{\mu: d_\mu(v)=1}\Bigl[f((\widetilde{(v_\mu^+)_\mu^-})_\mu)+ \sum_{\nu\neq\mu, d_\nu(v_\mu^+)=2}  f((\widetilde{(v_\mu^+)_\nu^-})_\nu)) \Bigr]
\]
\begin{equation}\label{eq:aminusg2} =\sum_{\mu: d_\mu(v)=1}\Bigl[f(\tilde{v}_\mu) + \sum_{ d_\nu(v)=2}  f((\widetilde{(v_\mu^+)_\nu^-})_\nu) \Bigr] 
=R_1f(v)+\sum_{\mu: d_\mu(v)=1}\sum_{ d_\nu(v)=2} f((\widetilde{(v_\mu^+)_\nu^-})_\nu) 
\end{equation}
since $d_\mu(v)=1$, $d_\nu(v_\mu^+)=2$ and $\nu\neq \mu$ imply $d_\nu(v)=2$.

For each fixed $\nu$ in the outer sum of 
\begin{equation}\label{eq:twisted_sum}
\sum_{d_\nu(v)=2}   \sum_{\mu: d_\mu(v)=1} f((\widetilde{(v_\mu^+)_\nu^-})_\nu) 
=\sum_{d_\nu(v)=2}   \sum_{\mu: d_\mu(v)=1} f((\widetilde{(v_\nu^-)_\mu^{+}})_\nu)  ,
\end{equation}
the vertices $v\in\Sigma_{p,q} $ and $\{(v_\nu^{-})_\mu^+:d_\mu(v)=1\}$ together form the set
of $(p,q)$-outer neighbors of $v_\nu^-\in \Sigma_{p+1,q-1}$ ($d_\mu(v)=1$ implies $\mu\neq \nu$). 
Thus, the inner sum on the right of (\ref{eq:twisted_sum}), without reflections, would amount to $-f(v)+(A_-f)(v_\nu^-)$.
Accounting for reflections, the inner sum (for fixed $\nu$) is $-f(\tilde{v}_\nu)+(\rho_\nu A_-f)(v_\nu^-)$.  
Summing over $\{\nu: d_\nu(v)=2\}=\{\nu: d_\nu(v_\nu^-)=1\}$, the expression  (\ref{eq:twisted_sum}) is then equal  to
\begin{equation}\label{eq:twistedsum2}
-(A_0 f)(v)+\sum_{\nu:d_\nu(v)=2} (\rho_\nu A_-f)(v_\nu^-)\, .
\end{equation}

Combining (\ref{eq:aminusg1}), (\ref{eq:aminusg2}),  and (\ref{eq:twistedsum2}) gives the following.
\begin{theorem} \label{thm:neutral_commutator_C5} Let $f$ be supported in $\Sigma_{p,q}$. 
Denote $A_-=A_{(p-1,q+1)\to (p,q)}$ and $A_+=A_{(p,q)\to (p-1,q+1)}$. Then 
\begin{equation}\label{eq:aminustwistedaplus} (A_-(A_{0} A_{+}-A_{+}A_{0})) f(v) =((R_1 - A_0) f)(v) + \sum_{\nu:d_\nu(v)=2} (\rho_\nu A_-f)(v_\nu^-)
\end{equation}
considered as operators on and between $\Sigma_{p,q}$ and $\Sigma_{(p-1,q+1)}$.
\end{theorem}
The formula (\ref{eq:aminustwistedaplus}) can be rewritten as $A_- \widetilde{A_+} f= (R_1-A_0) f+ \widetilde{A_+} A_- f$ where, on the left,
$A_-=A_{(p-1,q+1)\to (p,q)}$ and, on the right, $A_-=A_{(p,q)\to (p+1,q-1)}$. 
The composition $A_- \widetilde{A_+}$ simplifies in some cases.  For example, when  $f$ itself is in the kernel of $A_-$,
one also has $\rho_\nu A_- f=0$ for each $\nu $ so $ \sum_{\nu:d_\nu(v)=2} (\rho_\nu A_-f)(v_\nu^-)=0$ for each $v\in\Sigma_{p,q}$.
Consider next a case in which $f=A_+h$ where $h$ is supported in $\Sigma_{p+1,q-1}$.  Then 
\begin{equation}\label{eq:caseaplush} A_-f=A_-A_+h=[A_-,A_+]h+A_+A_- h =m(p+1,q-1)h +R_1h +(A_+A_-)h
\end{equation}
where $m(p,q)=(2(N-q)-3p)$ by Thm.~\ref{thm:commutator_C5}.

\bigskip

\begin{theorem}  \label{thm:adjacency_invariant_C5} Denote by $\mathcal{W}_{p,q,\lambda,0}$ the space of  vertex functions  on 
$\mathcal{C}_5^N$ supported in $\Sigma_{p,q}$ such that $f$ is in the kernel
of each subadjacency of $A_-$, is a $\lambda$-eigenvector of $R_1$, and is an eigenvector of $A_0$.  
Let $\mathcal{V}_{p,q,\lambda}$ consist of those vertex functions that are sums of
multiples of images of elements of $\mathcal{W}_{p,q,\lambda,0}$ under a sequence of compositions of  subadjacencies of $A_0$ and $A_+$.
Then $\mathcal{V}_{p,q,\lambda}$ is adjacency invariant.
\end{theorem}

\begin{proof}[Outline of Proof] If $\mathcal{V}_{p,q,\lambda}$ did not include applications of $A_0$ in its definition (only applications of $A_+$),
then the proof would essentially parallel that of Thm.~\ref{thm:adjacency_invariant_C3}.
By definition, $\mathcal{V}_{p,q,\lambda}$
is invariant under subadjacencies of $A_0$ and $A_+$.  So it suffices to show that it is invariant under subadjacencies of $A_-$.
The proof is by induction along similar lines as those of 
Thms~\ref{thm:adjacency_invariant_C3} and \ref{thm:adjacency_invariant_C4}. We argue that if a composition of $k$ subadjacencies of $A_+$ and $A_0$
gets mapped back to $ \mathcal{V}_{p,q,\lambda}$ by the appropriate subadjacencies of $A_-$, then so does any composition of
$k+1$ subadjacencies of $A_+$ and $A_0$. The base case ($k=1$) considers $A_- g$ where $g=A_+ f$ or $g=A_0 f$, $f\in \mathcal{W}_{p,q,\lambda,0}$.
The case $g=A_0 f$ is trivial since $f$ is an eigenvector of $A_0$ and in the kernel of $A_-$. The case $g=A_+f$ follows from 
Thm.~\ref{thm:commutator_C5} and the hypotheses that $R_1 f=\lambda f$ and $A_-f=0$.

Let $g_{r,s}$ be the component of $g\in \mathcal{V}_{p,q,\lambda}$ supported in $\Sigma_{r,s}$, $r\geq p$ and $s\geq q$.
We can assume that $g_{r,s}$ has one of the forms (i) $A_{(r-1,s)\to (r,s)} g_{r-1,s}$, (ii) $A_{(r+1,s-1)\to (r,s)} g_{r+1,s-1}$, or $A_0$ in the form
$A_{(r,s)\to (r,s)}$ applied to one of these forms. 
Consider the specific case in which $g_{r,s} = A_{(r-1,s)\to (r,s) } g_{r-1,s}$.
Then
\begin{multline*}
A_{(r,s)\to (r-1,s) } g_{r,s} = (A_{(r,s)\to (r-1,s) }  A_{(r-1,s)\to (r,s) }   -A_{(r-2,s)\to (r-1,s) }  A_{(r-1,s)\to (r-2,s) }  )   g_{r-1,s} \\
+
A_{(r-2,s)\to (r-1,s) }  A_{(r-1,s)\to (r-2,s) } g_{r-1,s}\\
=(C+
A_{(r-2,s)\to (r-1,s) }  A_{(r-1,s)\to (r-2,s) }) g_{r-1,s}\\
=((2(N-s)-3(r-1))I+R_1)  g_{r-1,s}+
A_{(r-2,s)\to (r-1,s) }  A_{(r-1,s)\to (r-2,s) }) g_{r-1,s}\\
=[(2(N-s)+3(r-1)+\lambda_{r-1,s}] g_{r-1,s}+A_{(r-2,s)\to (r-1,s) }  A_{(r-1,s)\to (r-2,s) }) g_{r-1,s}
\end{multline*}
by Thm.~\ref{thm:commutator_C5}.  Here, $\lambda_{r,s}$ is defined by $R_1  g_{r,s} =\lambda_{r,s} g_{r,s}$ according to the definition of $g$ as an image of a
fixed number of outer adjacencies applied to $
f\in\mathcal{W}_{p,q,\lambda,0}$, invoking Lem.~\ref{lem:eigen_reflection_5} in the corresponding number of instances, and  observing that $R_1$ commutes with $A_0=R_2$, the sum of reflections in level-two coordinates.
On the other hand, if we assume the induction hypothesis that any lower level components of 
$g\in \mathcal{V}_{p,q,\lambda}$
 are mapped back into $\mathcal{V}_{p,q,\lambda}$ under
inner adjacency, then $A_{(r-1,s)\to (r-2,s) }) g_{r-1,s}\in \mathcal{V}_{p,q,\lambda}$ and then so is $A_{(r-2,s)\to (r-1,s) }  A_{(r-1,s)\to (r-2,s) }) g_{r-1,s}$.
In this case, then, $A_{(r,s)\to (r-1,s) } g_{r,s}\in \mathcal{V}_{p,q,\lambda}$.
A corresponding argument applies to prove that  $A_{(r,s)\to (r+1,s-1)} g_{r,s}\in  \mathcal{V}_{p,q,\lambda}$. 

 To prove that 
$A_{(r,s)\to (r-1,s) } $ and $A_{(r,s)\to (r+1,s-1)} $ map $A_0 g_{r,s}$ back to $\mathcal{V}_{p,q,\lambda}$,
first  write $A_- A_0 g_{r,s}=A_-A_0 A_+ g_{r-1,s}=A_-[A_0A_+-A_+A_0]g_{r-1,s}+A_-A_+A_0 g_{r-1,s}$.
The term $A_-A_+A_0 g_{r-1,s}$ is shown to belong to $\mathcal{V}_{p,q,\lambda}$
along the same lines as above with $g_{r-1,s}$ replaced by $A_0 g_{r-1,s}$. 
The term $A_-[A_0A_+-A_+A_0]g_{r-1,s}$ is addressed using Thm.~\ref{thm:neutral_commutator_C5}.
Specifically, by (\ref{eq:aminustwistedaplus}),
\begin{equation}\label{eq:aminustwistedaplusgrs} 
(A_-(A_{0} A_{+}-A_{+}A_{0})) g_{r-1,s}(v) =((R_1 - A_0)  g_{r-1,s})(v) + \sum_{\nu:d_\nu(v)=2} (\rho_\nu A_-  g_{r-1,s})(v_\nu^-) \, .
\end{equation}
As above, $(R_1 - A_0)  g_{r-1,s}\in \mathcal{V}_{p,q,\lambda}$.  We can assume $g_{r-1,s}=A_+h$
for some $h\in \mathcal{V}_{p,q,\lambda}$ where 
$h$ is the image of a composition of strictly fewer outer or neutral adjacencies than $g_{r-1,s}$.  It follows from the induction hypothesis that $A_-A_+ h\in \mathcal{V}_{p,q,\lambda}$
(see also (\ref{eq:caseaplush})),
and then from the definition of $ \mathcal{V}_{p,q,\lambda}$  that  $\widetilde{A}_+ A_-  g_{r-1,s}$, the sum on the right in (\ref{eq:aminustwistedaplusgrs}), also belongs to 
$ \mathcal{V}_{p,q,\lambda}$, since $\widetilde{A}_+ =[A_0,\, A_+]$.
%
\end{proof}

\section{Application to spatio--spectral limiting\label{sec:ssl}}

The current work is motivated in part by an effort to develop a theory parallel to the Bell Labs theory of time and band limiting 
\cite{landau_pollak_II,landau_pollak_III,slepian_pollak_I,slepian_IV,slepian_V_1978}
 in the context of certain finite graphs, e.g.,  \cite{Hogan20172,hogan_lakey_gencube}.
The Bell Labs theory studies eigenfunctions and eigenvalues of the operators $P_\Omega Q_T$ where $(Qf)(t)=f(t)\mathbbm{1}_{[-T,T]}$
and $ (P_\Omega f)(t) =(\mathcal{F}^{-1} \mathbbm{1}_{[-\Omega/2,\Omega/2]}\mathcal{F} f)(t)$ where $(\mathcal{F} f)(\xi)=\int_{-\infty}^\infty f(t)\, e^{-2\pi i t\xi}\, dt$ is the Fourier transform on $\mathbb{R}$.
 The theory takes advantage, to an extent, of the self-dual nature of the Fourier transform on $\mathbb{R}$. 
 As finite abelian groups, $\mathbb{Z}_m^N$ also have self-dual Fourier transforms that can be realized as Kronecker powers of the matrix with $k$th column $e^{2\pi i j k/m}$, $j,k=0,\dots, m-1$.
 Spatio--spectral limiting operators on $\mathbb{Z}_m^N$ are analogues of $P_\Omega Q_T$, where an analogue $Q$ of $Q_T$ truncates spatially on $\mathcal{C}_m^N$ by
 multiplying a vertex function $f$ by the characteristic function of a symmetric neighborhood of the origin, and an analogue $P$ of $P_\Omega$ truncates
 in the graph Fourier transform domain by projecting onto the span of certain eigenvectors of the graph Laplacian on $\mathcal{C}_m^N$.
 It is typical to truncate to the space of vectors that are \emph{bandlimited} to eigenvalues below a certain threshold, e.g., \cite{pesenson_2008}, 
 but in the graph setting it can also be useful to truncate to other eigenspaces of the graph Fourier transform, 
 specifically in questions of recovery from (sub)samples, e.g. \cite{sandryhaila_etal_2015}.
 
 We consider a very specific example here of spatio--spectral limiting in $\mathcal{C}_5^4$, which has 625 vertices. We order the vertices first
 by increasing distance $d(v)$ to the identity element of $\mathbb{Z}_5^4$, then truncate by applying the diagonal matrix $Q$ equal to one
 on the indices of the vertices of distance at most 3 from the origin.  The truncated adjacency matrix is shown in Fig.~\ref{fig:adjacency_m5N4_full_partialv2}.
 We then restrict in the spectral domain also by applying the graph Fourier transform  $F$ (indexed with respect to the same ordering), 
 multiplying by the same matrix $Q$, and taking the inverse Fourier transform. We write $P=F^{-1} Q F$.
 
 The cardinalities of  $\Sigma_{p,q}$ ($p+2q\leq 3$) are listed in Tab.~\ref{tab:level_cardinality}.
 Altogether, there are $1+8+24+8+32+48=121$ vertices $v$ such that $d(v)\leq 3$. The range of $Q$ thus has dimension 121. The operator $P$
 is injective on this range, so the  self-adjoint operator $PQ$ has 121 linearly independent eigenvectors.  
 In Figs.~\ref{fig:m5N4K3_level_vectors}  and \ref{fig:m5N4K3_evecs} we plot the Fourier transforms of the eigenvectors 
 $PQ$  (equivalently, the eigenvectors of $ F PQ F^{-1}$) arranged by the base support $\Sigma_{p,q}$.  As explained in 
 \cite{hogan2018spatiospectral}, 
$FPQ$  is a polynomial in the adjacency operator and  these eigenvectors lie
 in spaces of the form $\mathcal{V}_{p,q,\lambda}$. The vectors were computed using the {\tt svd} function in {\tt matlab}.
  
 We refer to the vectors plotted in Fig.~\ref{fig:m5N4K3_level_vectors} as \emph{level vectors}. They are constant on the distance level 
 sets $\Sigma_{p,q}$.
 Their span is that of the vectors in the space $\mathcal{V}_{0,0}$ that are nonzero on $\Sigma_{0,0}$. 
 The dimension (six) is the number of distinct levels 
 (neutral adjacencies in Thm.~\ref{thm:adjacency_invariant_C5} preserve the property of being constant on each level), and thus the
 number of potential nonzero coefficients $c_k$ of a vector in $\mathcal{V}_{0,0}$ vanishing on $\{v: d(v)>3\}$.
 Consider next the vectors whose base support lies in $\Sigma_{1,0}$.
 Here there are six eigenspaces each of dimensions 3 and 4 (see Tab.~\ref{tab:eigenspaces}). 
 The number of eigenspaces of  $FPQ$  
  of each dimension is equal to the number of free coefficients of nonzero outer
 subadjacencies defining those $\mathcal{V}_{1,0}$ vectors vanishing on $\{v: d((v)>3\}$. 
 The dimension of each eigenspace is determined by the conditions of Thm.~\ref{thm:adjacency_invariant_C5}. 
 The set $\Sigma_{1,0}$ has eight vertices coming in coordinate pairs. 
 Eigenspaces of the reflection $R_1$ on vectors supported in $\Sigma_{1,0}$ thus have dimension four. Eigenvectors of $R_1$
 with eigenvalue $-1$ automatically have average value zero so are  orthogonal to constant vectors on $\Sigma_{1,0}$ (hence in kernel of $A_-$), 
 but for eigenvectors of $R_1$
 with eigenvalue $+1$, having average zero is an additional constraint on the coefficients of $g\in \mathcal{V}_{1,0,1}$ in its expansion in compositions of 
 $A_+$ and $A_0$ on $f\in \mathcal{W}_{1,0,1,0}$, in order to be orthogonal to constants. 
This implies that the $1$-eigenspaces of $R_1$ on $\Sigma_{1,0}$ in the kernel of $A_0$ can have dimension at most three.
 The dimensions of the eigenspaces of $FPQ$ based on the other levels $\Sigma_{p,q}$ listed in Tab.~\ref{tab:eigenspaces}
 can be explained along similar lines.  There are no eigenvectors with base support in $\Sigma_{0,1}$.

 \begin{tiny}
\begin{table}[tbhp]
{\footnotesize
\caption{\label{tab:level_cardinality}  Cardinalities of distance level sets}
}
\begin{center}
\bgroup
\def\arraystretch{1.3}
\begin{tabular}{|c |c|c|}
\hline 
Base& Cardinality&Indices   \\ \hline \hline
$\Sigma_{0,0} $ &1 & 1\\ \hline
$\Sigma_{1,0}$ &$\binom{4}{1}\times 2=8$& 2--9\\ \hline
$\Sigma_{2,0} $ &$\binom{4}{2}\times4=24$&10--33\\ \hline
$\Sigma_{0,1} $ &$\binom{4}{1}\times 2=8$&34--41\\ \hline
$\Sigma_{3,0} $ & $\binom{4}{3}\times 8=32$&42--73\\ \hline
$\Sigma_{1,1} $ &$\binom{4}{1}\times \binom{3}{1}\times 4=48$ &74--121\\ \hline
\hline
\end{tabular}
\egroup
\end{center}
\end{table}%
\end{tiny}
  
\begin{tiny}
\begin{table}[tbhp]
{\footnotesize
\caption{\label{tab:eigenspaces}  Eigenspaces of $PQ$, $m=5$, $N=4$, $K=3$}
}
\begin{center}
\begin{tabular}{|c |c|c|c|}
\hline 
Base& Dim (\#)& $R_1$ eigenvalue  & eigenvector indices \\ \hline \hline
$\Sigma_{0,0} $ &1 (6) &0  & 1,9,29,92,93,121 \\ \hline
$\Sigma_{1,0}$ &4 (6) &$ -1$&5-8,30-33,37-40,59-62,110-113,117-120\\ \hline
$\Sigma_{1,0} $ &3 (6)& 1&2-4,10-12,49-51,56-58, 89-91,114-116\\ \hline
$\Sigma_{2,0} $ &6 (2) & $-2$& 23-28,96-101\\ \hline
$\Sigma_{2,0} $ &8 (3) &0 & 15-22,41-48,102-109\\ \hline
$\Sigma_{2,0} $ &2 (2) & 2&13-14,94-95\\ \hline
$\Sigma_{3,0} $ &4 (1) &$-3$ & 52-55\\ \hline
$\Sigma_{3,0} $ &18 (1) &$-1$ &63-80\\ \hline
$\Sigma_{3,0} $ &8 (1) &1 &81-88\\ \hline
$\Sigma_{1,1} $ &3 (1) &$-1$&34-36\\ \hline
\hline
\end{tabular}
\end{center}
\end{table}%
\end{tiny}


  \begin{figure}[tbhp]
\centering 
\includegraphics[width=\textwidth,height=2.5in]{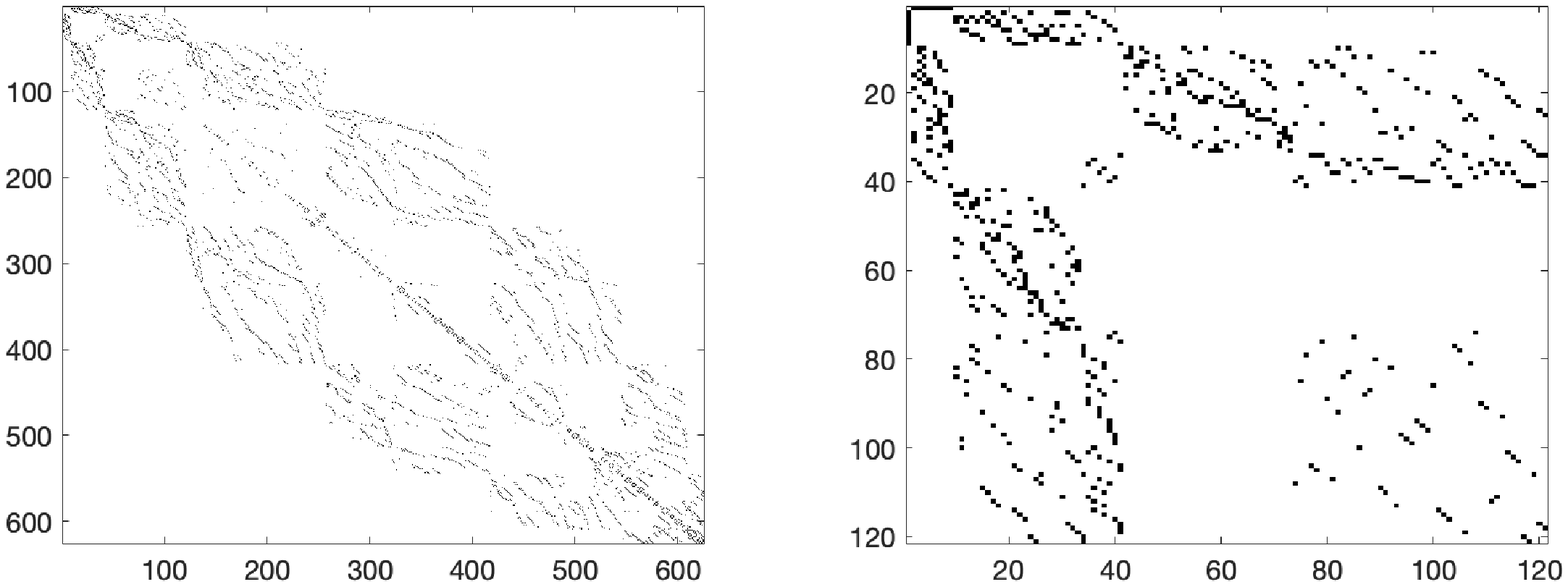}
\caption{\label{fig:adjacency_m5N4_full_partialv2} 
Adjacency matrix, $m=5$, $N=4$ (left) and principal minor of size 121 (right). Vector indices are ordered by increasing level $p+2q$, and by increasing $q$ for fixed $p+2q$.
}
\end{figure}

  \begin{figure}[tbhp]
\centering 
\includegraphics[width=\textwidth,height=2.5in]{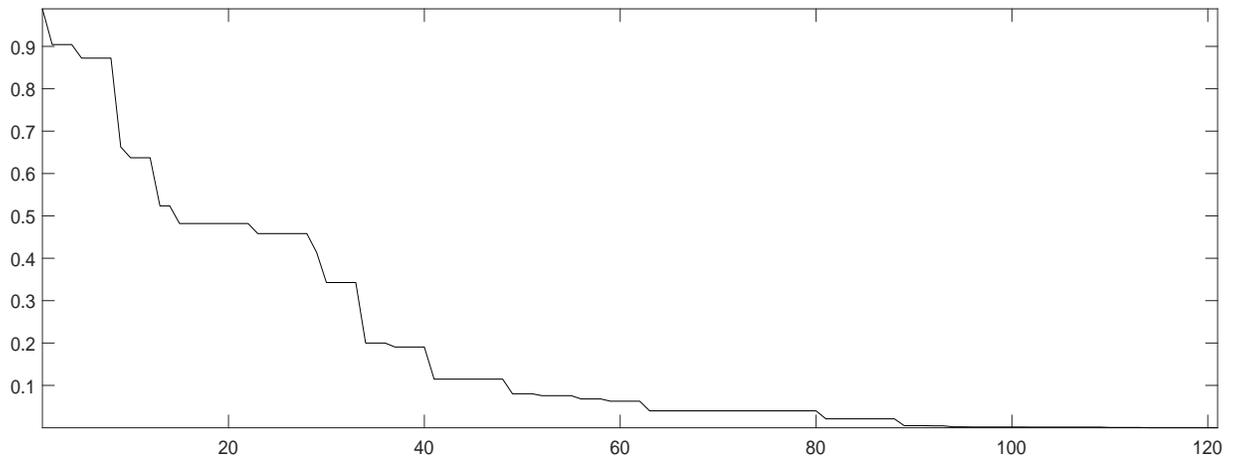}
\caption{\label{fig:m5N4K3_evals}
Nonzero eigenvalues of $PQ$, $m=5$, $N=4$, $K=3$. Intervals of fixed eigenvalue conform to indices of the eigenspaces listed in Tab.~\ref{tab:eigenspaces}}
\end{figure}

  \begin{figure}[tbhp]
\centering 
\includegraphics[width=\textwidth,height=2.5in]{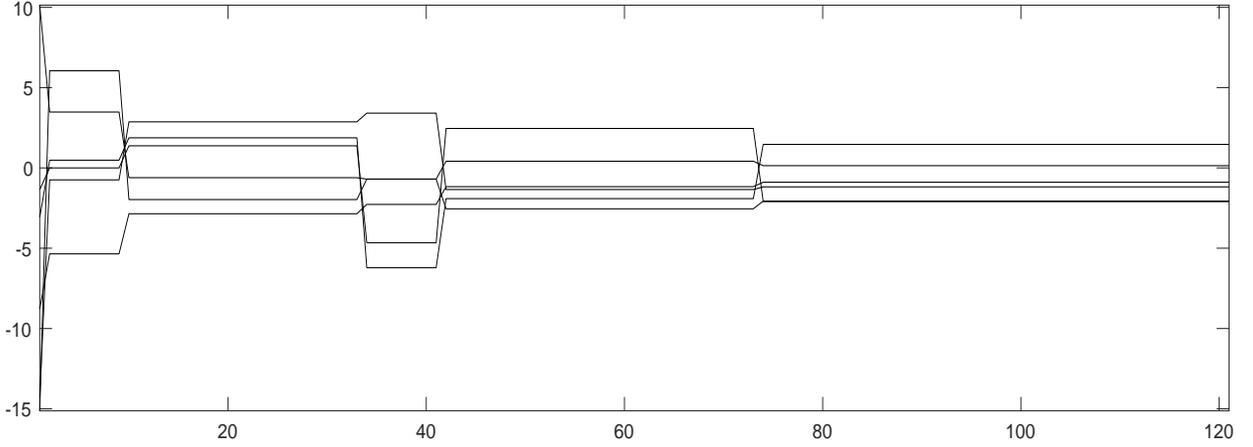}
\caption{\label{fig:m5N4K3_level_vectors}
Real parts of eigenvectors of the matrix $FPQ$, $m=5$, $N=4$, $K=3$. These are the vectors listed in Tab.~\ref{tab:eigenspaces} with base $\Sigma_{0,0}$.
They are constant on the level sets $\Sigma_{p,q}$}
\end{figure}

  \begin{figure}[tbhp]
\centering 
\includegraphics[width=\textwidth,height=3.5in]{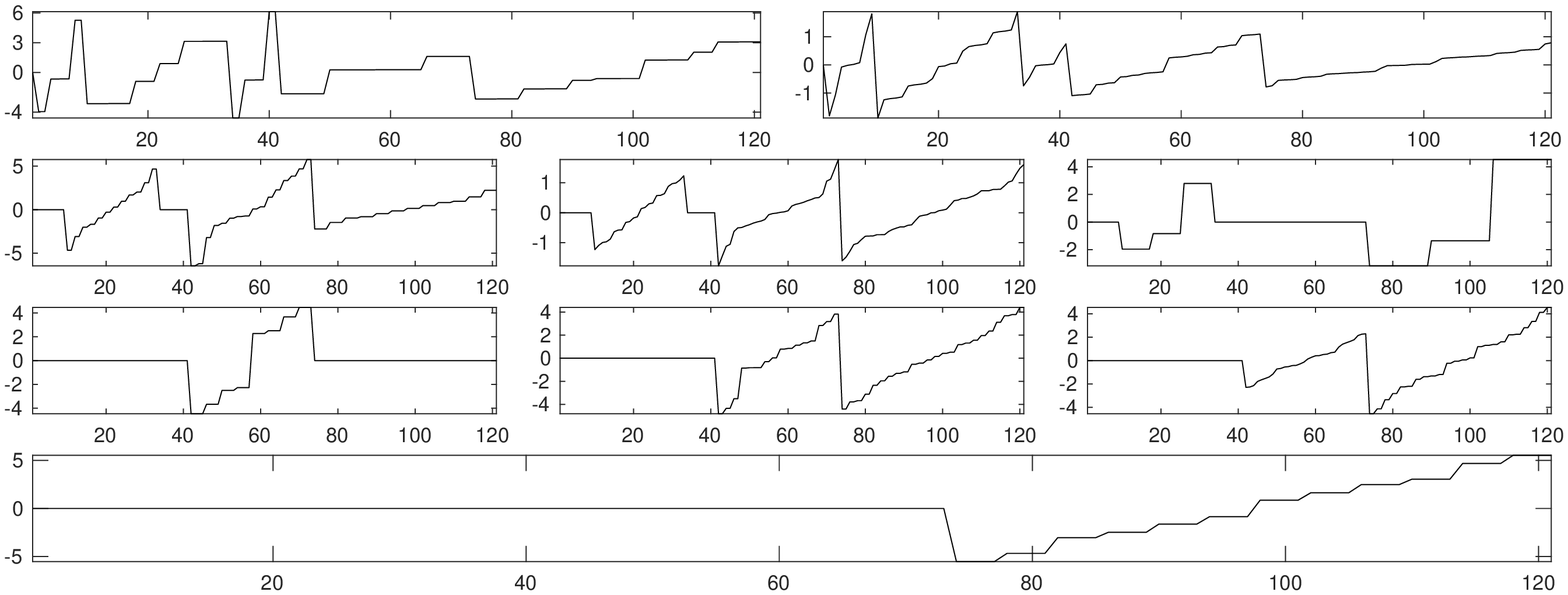}
\caption{\label{fig:m5N4K3_evecs} 
Partially sorted eigenvectors of the matrix $FPQ$, $m=5$, $N=4$, $K=3$. Real part of one eigenvector is plotted from each group listed in Tab.~\ref{tab:eigenspaces} 
corresponding to base level greater than zero. Values are sorted within each level set to make value symmetries more evident}
\end{figure}

\section{Conclusions\label{sec:conclusions}}
We have described certain spaces of vertex functions on products of cycles $\mathcal{C}_m^N$ for $m=3,4,5$ that are invariant under
the adjacency operators on these Cayley graphs.  In the cases $m=3,4$ these spaces are observed to be isomorphic to spaces of the 
form $\mathcal{W}_{L}\times \mathbb{C}^M$ where $\mathcal{W}_{L}$ is a space of vectors supported on a set $L$ of fixed path distance
to the identity of $\mathbb{Z}_m^N$, and the action of the adjacency operator can be expressed solely on the $\mathbb{C}^M$ factor.
The techniques can be extended to describe corresponding adjacency-invariant spaces on $\mathcal{C}_m^N$ for larger $N$ in terms of
inner- and-outer subadjacencies, and level-one and level $m/2-1$ reflections ($m$-even), or level-one reflections and level $(m-1)/2$ neutral adjacencies ($m$ odd).
Completeness properties---how to decompose the restriction of a vertex function to a neighborhood of the identity into a sum of terms, each of which
lies in one of the invariant spaces described here---will be addressed in future work 
(cf., \cite{hogan2018spatiospectral} for the $\mathbb{Z}_2^N$ case).

The results presented here, of course, do not readily extend to broad families of graphs. The methods should extend, with further complication, to powers of other small
graphs having ample symmetry, e.g., \cite{hogan_lakey_gencube}.

 \thispagestyle{empty}
 
 \bibliographystyle{amsplain}
\bibliography{boolean_eigenspace_refs.bib}

\end{document}